\documentclass[letterpaper, 10 pt, conference]{ieeeconf}  
\usepackage{amsmath}
\usepackage{amssymb}
\IEEEoverridecommandlockouts                              
\usepackage{graphics} 
\usepackage{epsfig} 
\usepackage{epstopdf}
\usepackage{mathptmx} 
\usepackage{times} 
\usepackage{amsmath} 
\usepackage{amssymb}  
\usepackage{subfigure}
\usepackage{breqn}
\usepackage{bbm}
\usepackage{dsfont}

\newtheorem{defn}{Definition}
\newtheorem{thm}{Theorem}

\newtheorem{lem}{Lemma}

\newtheorem{cor}{Corollary}

\newcommand{\mcl}[1]{\mathcal{ #1}}

\newcommand{\R}{\mathbb{R}}

\newcommand{\N}{\mathbb{N}}

\title{\LARGE \bf
	Using SOS for Optimal Semialgebraic Representation of Sets: \\Finding Minimal Representations of Limit Cycles, Chaotic Attractors and Unions}

\author{Morgan Jones%
	\thanks{M. Jones is with the School for the Engineering of Matter, Transport and Energy, Arizona State University, Tempe, AZ, 85298 USA. e-mail: {\tt \small morgan.c.jones@asu.edu } },
	Matthew M. Peet
	\thanks{M. Peet is with the School for the Engineering of Matter, Transport and Energy, Arizona State University, Tempe, AZ, 85298 USA. e-mail: {\tt \small mpeet@asu.edu } }
}

\begin{document}

		\maketitle
		\thispagestyle{empty}
		\pagestyle{empty}
\begin{abstract}
 In this paper we show that Sum-of-Squares optimization can be used to find optimal semialgebraic representations of sets. These sets may be explicitly defined, as in the case of discrete points or unions of sets; or implicitly defined, as in the case of attractors of nonlinear systems. We define optimality in the sense of minimum volume, while satisfying constraints that can include set containment, convexity, or Lyapunov stability conditions. Our admittedly heuristic approach to volume minimization is based on the use of a determinant-like objective function. We provide numerical examples for the Lorenz attractor and the Van der Pol limit cycle.
\end{abstract}

\section{Introduction}

In this paper we consider nonlinear Ordinary Differential Equations (ODE's) of the form
\begin{equation} \label{ode}
\dot{x}(t)=f(x(t)), \quad x(0)=x_0.
\end{equation}
where $f: \R^n \to \R^n$ is the vector field and $x_0 \in \R^n$ is the initial condition.  A set $A \in \R^n$ is an attractor of System \eqref{ode} if for any solution $x(t)$, there exist a $T>0$ such that $x(t)\in A$ for all $t > T$. For a given polynomial $f$, our goal is to use SOS and polynomial Lyapunov functions to parameterize and optimize over the set of attractors while minimizing volume of the attractor.

Attractors capture the long-term properties of systems and can be thought of as a minimal notion of stability for chaotic systems with no stable limit cycles or equilibrium point.  The first and most famous example of a chaotic system with an attractor was proposed in 1963 by E.N. Lorenz to model convection rolls in the atmosphere. The Lorenz attractor \cite{Colin_1982}, with its distinctive butterfly shape, contains three equilibrium points and is of zero volume \cite{Strogatz_1994}. While these ``Chaotic attractors'' have been shown within the purview of chaos theory to limit the ability of models to predict future events~\cite{Lee_2016}, identification of a minimal attractor can improve our ability to understand chaotic systems by bounding the domain on which determinism fails.

A Lyapunov function, $V$, is a generalization of the notion of energy and given a Lyapunov function, we can bound a stable attractor, $A$, by identification of the maximum energy of any point in the attractor $\gamma = \sup_{x \in A}V(x)$. Then the attractor is contained in the level set $L(V,\gamma) :=\{x \in \R^n \;:\; V(x) \le \gamma\}$. This approach allows one to identify invariant subsets of an ODE and in~\cite{Li_2004}, a quadratic Lyapunov function was used to show the Lorenz attractor is contained in an ellipsoid of finite major axis. Meanwhile, in~\cite{Yu_2005} the estimate of the Lorentz attractor was improved through the use of a non-quadratic Lyapunov function.
A further refined bound on the Lorentz attractor was given in~\cite{Wang_2012} where Lyapunov methods were used to seed the initial approximation of a time advecting algorithm. Most recently, in unpublished work~\cite{Goluskin_2018}, a heuristic reduction of Putinar's Positivstellensatz to scalar multipliers was proposed for the purpose of using Sum-of-Squares (SOS) optimization to find polynomial Lyapunov functions which bound the Lorentz attractor. However, this work was unable to provide a metric for minimizing the volume of the resulting attractor.



\textbf{Sublevel Set Volume Minimization}

Set approximations has many practical applications. For instance, the F/A-18 fighter jet is susceptible to an unstable oscillation called ``falling leave mode''. In \cite{Chakroborty_2011_falling_leaf_mode} a sublevel set inner approximate of the region of attraction (ROA) for the dynamics of F/A-18 fighter jet was found. The use of sublevel sets of SOS polynomials for set estimation of ROA's for aircraft dynamics were also used in \cite{Chakroborty_2011} and \cite{Khrabrov_2017}. In \cite{Ghaoui_2001_ellipse_state_est} \cite{Durieu_2001_ellipsoidal_state_est} ellipsoidal sets are used as outer set approximation for state estimation. Zonotope sets were used as an outer approximate in \cite{Alamo_2005_zonotope_state_est} for state estimation and \cite{Ingimundarson_2009_zonotopes_fault_detection} for fault detection. In \cite{Hwang_2003_polytopic_reachable_sets} polytopic sets were used as outer approximations of reachable sets.

In this paper we construct an outer approximation of a set by constraining a sublevel set, of the form $L(V,1)$, to contain the set. Furthermore we heuristically minimize the volume of $L(V,1)$ to improve the approximation. The proposed set approximation method in this paper can be broken down into two cases. The first case is when the set to be approximated is explicitly defined as unions of semialgebriac sets. The second case is when the set to be approximated is implicitly defined as an attractor of an ODE of Form \eqref{ode}. These two cases can each be formulated as an optimization problem that has an objective function related to the volume of the outer set approximate and constraints ensuring the set is contained in the outer approximate. In the first case the unions of semialgebriac sets are constrained to be contained inside the outer approximate using Putinar's Positivestellensatz, whereas in the second case the attractor is constrained to be contained inside the outer approximate using Lyapunov theory. In both cases to make the optimization problem tractable we consider sets that can be written as sublevel sets of Sum-of-Squares (SOS) polynomials, $L(V,1)$ where $V(x)=z_d(x)^T P z_d(x)$, $P>0$ is a positive definite matrix and $z_d(x)$ is a vector of monomials of degree $d$ or less. We note that increasing the value of the eigenvalues of $P$ increases the value of $V(x)$ for all $x \in \R^n$ and thus the volume of $L(V,1)$ is reduced. Since $\det P$ is the product of the positive eigenvalues of $P$, we propose to minimize the convex objective function $-\log \det P$ to reduce the volume of $L(V,1)$. Furthermore in \cite{Boyd_1994} it is shown in the case of $d=1$ that $\det P^{-1}$ is proportional to the volume of $L(V,1)$.

Volume minimization of sublevel sets is a difficult problem. Using $\log \det$ functions as a metric for volume of a sublevel set of an SOS polynomial was first proposed in~\cite{Magnani_2005}, although that work only treated explicit constraints generated by containment of a set of points. In~\cite{Ahmadi_2017}, this approach was extended to containment of intersections of semialgebriac sets. In \cite{Dabbene_2017} an optimization problem with an objective function involving trace is constructed. In the current paper, we retain the $\log \det $ objective, but provide a more rigorous justification, while extending use of this approach to include estimating attractor sets and representing the union of semialgebraic sets.


The rest of this paper is as follows; in Section \ref{Section: Outer Approximation of sets} we formulate a convex optimization problem that is solved by the heuristic best outer approximation of some semialgebriac set. In Section \ref{sec: Numerical examples of set containment} we present numerical examples of our minimum volume set containment algorithm. In Section \ref{sec: outeer approx of attractors} we formulate the problem of computing the outer approximate of an attractor as an optimization problem involving volume minimization. In Section \ref{sec: Numerical Examples of Attractor Approximation} we present our numerical examples of our attractor approximates for the Van der Poll system and Lorenz system. Finally we give our conclusion in Section \ref{sec:conclusion} and future work in \ref{sec:future work}.


\section{Notation}
For a set $A \subset \R^n$ we define the indicator function $\mathds{1}_A : \R^n \to \R$ by $\mathds{1}_A(x) = \begin{cases}
 & 1 \text{ if } x \in A\\
& 0 \text{ otherwise}
\end{cases}$. For a set $A \subset \R^n$ we define $vol\{A\}= \int_{\R^n} \mathds{1}_{A}(x) dx$. We denote the power set of $\R^n$ by $P(\R^n)=\{X:X\subset \R^n\}$. For two sets $A,B \subset \R^n$ we denote $A/B=\{x \in A: x \notin B\} $. If $M$ is a subspace of a vector space $X$ we denote equivalence relation $\sim_M$ for $x,y \in X$ by $x \sim_M y$ if $x-y \in M$. We denote quotient space by $X \pmod M=\{ \{y \in X: y \sim_M x \}: x \in X\}$. For a function $V: \R^n \to \R$ and a scalar $\alpha>0$ we define the $\alpha$-sublevel set by $L(V, \alpha):=\{x \in \R^n: V(x) < \alpha \}$. We say $S \subset \R^n$ is a semi-algebraic set if it can be written in the form $S = \{x \in \R^n: g_1(x) \le 0,..., g_m(x) \le 0 \}$ for some functions $g_i: \R^n \to \R$ for $i \in \{1,...,m\}$. We denote $S^n_{++}$ to be the set of positive definite $n \times n$ matrices. We denote the set $GL(n, \R)$ to be the set of invertible $n \times n$ matrices with real elements. For a matrix $M \in GL(n,\R)$ we define an ellipse $\mcl E_M := \{ x \in \R^n: x^T M^T M x \le 1 \} \in P(\R^n)$. For $x \in \R^n$ we denote $z_d(x)$ to be the vector of monomial basis in $n$-dimensions with maximum degree $d \in \N$. We say the polynomial $p :\R^n \to \R$ is Sum-of-Squares (SOS) if there exists polynomials $p_i :\R^n \to \R$ such that $p(x) = \sum_{i=1}^{k} (p_i(x))^2$. We denote $\sum_{SOS}$ to be the set of SOS polynomials.


\section{Outer Approximation of Sets} \label{Section: Outer Approximation of sets}
In this paper we would like to compute an outer approximations of a set. In general we formulate an optimization problem of the form,
\begin{align} \label{opt: general set countainment}
\min_{X \in C} \{ & D(X,Y) \}\\ \nonumber
& \text{subject to: } Y \subseteq X \nonumber
\end{align}
where $Y \subset \R^n$, $C \subset P(\R^n)$ and $D:P(\R^n) \times P(\R^n) \to \R$ is some metric that measures the distance between two subsets of $\R^n$.
\subsection{Volume is a Metric for Set Approximation} \label{section: A metric for volume}
In this section we propose a metric that can be used in the optimization problem \eqref{opt: general set countainment} based on the volume of a set, $vol\{A\}= \int_{\R^n} \mathds{1}_{A}(x) dx$ where $A \subset \R^n$.

\begin{defn}
	 $D: X \times X \to \R$ is a \textit{metric} over some set $X$ if $D$ satisfies the following properties $\forall x.y \in X$,
	\begin{itemize}
		\item $D(x,y) \ge 0$,
		\item $D(x,y)=0$ iff $x=y$,
		\item $D(x,y)=D(y,x)$,
		\item $D(x,z) \le D(x,y) + D(y,z)$.
	\end{itemize}
\end{defn}

\begin{lem}
	Consider the quotient space, 
\[
C:=P(\R^n) \pmod{\{X \in P(\R^n): X \ne \emptyset, vol \{ X\} =0 \}}.
\] Then the function $D: C \times C \to \R$,
	\begin{equation}
	D_V(X,Y)= vol\{ (X/Y) \cup (Y/X) \}
	\end{equation}
	defines a metric over $C$.
\end{lem}
\begin{proof}
	It is clear that $D(X,Y) \ge 0$ and $D(X,Y)=D(Y,X)$ for all $X,Y \in C$. Property $D(X,Z) \le D(X,Y) + D(Y,Z)$ follows from the relationship: $X/Z\subseteq X/Y \cup Y/Z$ for any $X,Y,Z \in P(\R^n)$. Finally $D(X,Y)=0$ iff $X=Y$ follows from properties of the quotient space $C$.
\end{proof}


In this paper we have chosen to use the metric $D_V$ in optimization problems of the form \eqref{opt: general set countainment}. This is because, as we will see in the next section, there is a relationship between $D_V$ and the determinant of some matrix $M$ when the constraint set, $C$, is of the form $C = \{ \mcl E_M: M \in GL(n,\R) \}$.

\begin{lem} \label{lem: D_V is related to vol}
	If $Y \subseteq X$ then $D_V(X,Y)= vol\{X\}-vol\{Y\}$.
\end{lem}
\begin{proof}
	\begin{align*}
	D_V(X,Y) & = vol\{X/Y \cup Y/X \}\\
	& = vol\{X/Y\} + vol\{Y/X\}\\
	&= vol\{X/Y\}\\
	&= vol\{X\} - vol\{Y\}
	\end{align*}	
	where the second equality is because $X/Y$ and $Y/X$ are disjoint sets; the third equality is because $Y/X= \emptyset$ as $Y \subseteq X$; the fourth equality is because $Y \subseteq X$ so $Y \cap X= Y$ and thus $\mathds{1}_{X/Y}(x)= \mathds{1}_{X}(x)- \mathds{1}_{X \cap Y}(x) =  \mathds{1}_{X}(x)- \mathds{1}_{Y }(x)$ for all $x \in \R^n$, which implies $vol\{X/Y\}=vol\{X\}-vol\{Y\}$.
\end{proof}
In the next corollary we show under the metric $D_V$ the optimization problem \eqref{opt: general set countainment} is equivalent to the optimization problem,
\begin{align} \label{opt: general set countainment 2}
\min_{X \in C} \{ & vol\{X \} \}\\ \nonumber
& \text{subject to: } Y \subseteq X \nonumber
\end{align}
where $Y \subset \R^n$ and the constraint set is of the form $C \subset P(\R^n)$.
\begin{cor} \label{cor: opt involing D_V and vol are equivalent}
	If $X_1^*$ solves \eqref{opt: general set countainment} for the metric $D_V$ and $X_2^*$ solves \eqref{opt: general set countainment 2} then $X_1^*=X_2^*$.
\end{cor}
\begin{proof}
	The set of feasible solutions, $\{X \in \R^n : X \in C \text{ and } Y \subset X\}$, for the optimization problem \eqref{opt: general set countainment} is equal to the set of feasible solutions for \eqref{opt: general set countainment 2}. Moreover Lemma \ref{lem: D_V is related to vol} shows that minimizing the objective function in \eqref{opt: general set countainment} is equivalent to minimizing the objective function in \eqref{opt: general set countainment 2} as the two functions only differ by a constant.
	\end{proof}

\subsection{How We Minimize the Volume of a Set}
Corollary \ref{cor: opt involing D_V and vol are equivalent} shows how the outer-approximation of a set, formulated in the optimization problem \eqref{opt: general set countainment}, is equivalent to minimizing the volume of the outer set, formulated in the optimization problem \eqref{opt: general set countainment 2}. However evaluating and minimizing the volume of a set is difficult. In this section we seek to make this problem tractable. Here we will show if the constraint set, $C$ in \eqref{opt: general set countainment 2}, is the set of ellipses then there exists an equivalent convex optimization problem. To formulate this convex optimization problem we first must understand the relationship between the determinant and volume.

The determinant can be understood as the ratio between the volumes of a set and a linear transformation of that set. To prove this property of the determinant one can use the formula for integration by substitution given next.

\begin{thm}[Theorem 7.26 \cite{Rudin_1987}] \label{thm: integration by sub}
	Let $U$ be an open set in $\R^n$ and $\phi : U \to \R^n$ an injective differentiable function with continuous partial derivatives, the Jacobian of which is nonzero for every $x \in U$. Then for any real-valued, compactly supported, continuous function $f$, with support contained in $\phi(U)$,
	\begin{equation} \label{eqn: int by sub}
	{\displaystyle \int _{\phi (U)}f(\mathbf {v} )\,d\mathbf {v} =\int _{U}f(\phi (\mathbf {u} ))\left|\det(\operatorname {D} \phi )(\mathbf {u} )\right|\,d\mathbf {u} }
	\end{equation}
	where $\det(\operatorname {D} \phi )(\mathbf {u} )$ denotes the determinant of the Jacobian matrix of the function $\phi$ at $\mathbf u$.
\end{thm}
We now prove the relationship between the volume and determinant by selecting particular functions in \eqref{eqn: int by sub}.
\begin{cor} \label{cor: general det ratio}
Let $U$ be an open set in $\R^n$ then,
\begin{equation} \label{eqn: general det}
|\det(A)| = \frac{ vol\{A x: x \in U \}}{vol\{U\}}.
\end{equation}
\end{cor}
\begin{proof}
Let us consider the function $\phi: U \to \R^n$ and $f: \R^n \to \R$	defined by $\phi(x)= Ax$ and $f(x)= \mathds{1}_{\phi(U)}(x)$ respectively. Applying this to \eqref{eqn: int by sub} we get,
\begin{align*}
\int \mathds{1}_{\phi(U)}(x) dx = \int \mathds{1}_{U}(x) |\det(A)| dx.
\end{align*}
Now by rearranging the above equality we get \eqref{eqn: general det}.
\end{proof}

In this paper we are interested in the special case of Corollary \ref{cor: general det ratio} when $U$ is a sublevel set. Specifically it can be shown for an invertible square matrix $A \in GL(n,\R)$ and a function $g: \R^n \to \R$, such that $L(g,1)$ is open, \eqref{eqn: general det} becomes,
\begin{equation} \label{eqn: det is ratio of volumes}
\det(A)= \frac{ vol\{L(g \circ A^{-1},1 )\}}{vol\{L(g,1)\}}.
\end{equation}
Equation \eqref{eqn: det is ratio of volumes} relates the volume of sublevel sets and the determinant of invertible matrices. We will use this equation to justify how optimization problems with determinants in the objective function minimize volumes of sublevel sets. To do this let us consider the optimization problem \eqref{opt: general set countainment 2} when the constraint set, $C$, is the set of ellipses,
\begin{align} \label{opt: general elipse optimization problem}
\min_{X \in \{\mcl E_M: M \in GL(n,\R)\}}  & vol\{X \}\\ \nonumber
& \text{subject to: }Y \subseteq X \nonumber
\end{align}
where $Y=\{b_1,....,b_m \}$.

Let us also consider the following optimization problem,
\begin{align} \label{opt: elipse optimization problem}
 \min_{A \in S^n_{++}} \{ & -\log \det A \}\\ \nonumber
& \text{subject to: } b_i^T A b_i \le 1 \text{ for } i \in \{1,...,m\} \nonumber
\end{align}
where $b_i \in \R^n$.

We will show in Lemma \ref{lem: opt minimizes elipses and constrains points} that the optimization problems \eqref{opt: general elipse optimization problem} and \eqref{opt: elipse optimization problem} are equivalent. Furthermore in Lemma \ref{lem: logdet is convex} we will show \eqref{opt: elipse optimization problem} is convex.

\begin{lem} \label{lem: opt minimizes elipses and constrains points}
Suppose $\mcl E_{M^*}$ solves \eqref{opt: general elipse optimization problem} and $A^{*}$ solves \eqref{opt: elipse optimization problem} then $A^{*}={M^*}^T M^*$.
\end{lem}
\begin{proof}
 Let us denote $A ={M^*}^T M^*$. We first show that $A$ is feasible for \eqref{opt: elipse optimization problem}. Since by the constraints of \eqref{opt: general elipse optimization problem} we have $\{b_1,....,b_m \} \subset \mcl E_{M^*}$ it follows $b_i^T {M^*}^T M^* b_i \le 1 \text{ for } i \in \{1,...,m\}$. That is $b_i^T A b_i \le 1 \text{ for } i \in \{1,...,m\}$. Moreover it is clear $A \in S^n_{++}$ since $A$ is the matrix multiplication of $M^*$ with itself and thus $A$ is feasible for \eqref{opt: elipse optimization problem}.

 Suppose $A^{*}$ solves \eqref{opt: elipse optimization problem}. Since $A^* \in S^n_{++}$ there exists $M \in GL(n, \R)$ such that $A^*= M^T M$. We will now show $M$ is feasible for \eqref{opt: general elipse optimization problem}. By the constraints of \eqref{opt: elipse optimization problem} we have we have $b_i^T A b_i \le 1 \text{ for } i \in \{1,...,m\}$ so $b_i^T {M}^T M b_i \le 1 \text{ for } i \in \{1,...,m\}$. Thus it now follows $\{b_1,....,b_m \} \subset \mcl E_{M}$.

 We will now show that minimizing the objective function of \eqref{opt: general elipse optimization problem} is equivalent to minimizing the objective function of \eqref{opt: elipse optimization problem} over the same constraint set. By writing the n-dimensional ball in the form $B(0,1):=\{x : x^T x \le 1 \}=L(g,1)$ where $g(x)=x^T x$ we can write $\mathcal{E}_{M}= L( g \circ M , 1)$ for $M \in GL(n,\R)$. It can now be shown the objective function of \eqref{opt: general elipse optimization problem} is such that,
\begin{align*}
vol\{\mcl E_{M} \} & = \det({M}^{-1}) vol\{B(0,1)\} \\
& = \frac{\pi^{\frac{n}{2}} }{\Gamma(\frac{n}{2} +1)} \sqrt{\det({A}^{-1})},
\end{align*}

\noindent where the first equality follows by properties of the determinant \eqref{eqn: det is ratio of volumes}; the second equality follows because $vol\{B(0,1)\}= \frac{\pi^{\frac{n}{2}} }{\Gamma(\frac{n}{2} +1)}$, where $\Gamma$ is the Gamma function, and commutative properties of the determinant, where $A=M^TM$. Thus minimizing $vol\{\mcl E_M \}$ is equivalent to minimizing $\sqrt{\det(A^{-1})}$ which is equivalent to minimizing the objective function in \eqref{opt: elipse optimization problem} since $\log \{ \sqrt{\det( A^{-1} )} \}=0.5 \log \{ (\det A)^{-1} \} = -0.5\log \det A$.
\end{proof}
The optimization problem \eqref{opt: elipse optimization problem} is a convex optimization problem as the constraints are affine in the decision variable, $A \in S_{++}^n$, and the objective function is convex; stated in the following Lemma.
\begin{lem}[\cite{Boyd_2004}] \label{lem: logdet is convex}
	The function $f: S^n_{++} \to \R$ given by $f(X)=-\log \det (X)$ is convex.
\end{lem}

\subsection{Heuristic Volume Minimization of Sublevel Sets of SOS Polynomials}
In the optimization problem \eqref{opt: general elipse optimization problem} the decision variable, $X \subset \R^n$, is constrained to be an ellipse. Equivalently we can also think of $X$ being constrained to be the sublevel set of a quadratic polynomial, of the form $L(x^T Ax,1)$ where $A \in S^n_{++}$. Naturally we would like to expand the type of sets that our outer approximation can take. One way to do this is to expand the constraint set to include sublevel sets created by non-quadratic polynomials.

Inspired by \eqref{opt: elipse optimization problem} we next give an optimization problem that heuristically minimizes the distance between a set $Y= \{b_1,..., b_m\}$ and a sublevel set of the $L(V,1)$ where $V(x)=z_d(x)^T A z_d(x)$, $A \in S^N_{++}$ and $N=\dim \{z_d\} $.

\begin{align} \label{opt: monomials to constrain points}
 \min_{A \in S^N_{++}} \{ & -\log \det A \}\\ \nonumber
& \text{subject to: } z_d(b_i)^T A z_d(b_i) \le 1 \text{ for } i \in \{1,...,m\}. \nonumber
\end{align}
We see that the optimization problem \eqref{opt: elipse optimization problem} is a special case of \eqref{opt: monomials to constrain points} when $d=1$. However allowing for $d>1$ the 1-sublevel set of $V(x)=z_d(x)^T A z_d(x)$ is able to form more complicated shapes.

To understand heuristically why a solution of \eqref{opt: monomials to constrain points} can construct a solution close to \eqref{opt: general set countainment 2} we note that increasing the eigenvalues of $A \in S_{++}^n$ also increases the value of the function $V(x) = z_d(x)^T A z_d(x)$ at every $x \in \R^n$. This results in less $x \in \R^n$ such that $V(x) <1$. Thus the volume of sublevel set $L(V,1)$ is reduced.


As argued in \cite{Ahmadi_2017} there is another way to see how the optimization problem \eqref{opt: monomials to constrain points} heuristically minimizes the volume of the 1-sublevel set of $V(x)=z_d(x)^T A z_d(x)$, while constraining the 1-sublevel set to contain $\{b_i\}_{i=\{1,...,m\}}$. Let us define,
\begin{align*}
& \mcl M_A = \{x \in \R^n: z_d(x)^T A z_d(x) \le 1 \}\\
& T_1= z_d(\R^n):= \{z_d(x) \in \R^N: x \in \R^n\}\\
& T_2= \{y \in \R^N : y^T A y \le 1 \}
\end{align*}
We now have the identity $z_d( \mcl M_A)=T_1 \cap T_2$; where $z_d( \mcl M_A)= \{z_d(x): x \in \mcl M_A\}$. The optimization problem \eqref{opt: monomials to constrain points} constrains $b_i \in \mcl M_A$ for all $i \in \{1,...,m\}$ and by Lemma \ref{lem: opt minimizes elipses and constrains points} minimizes $vol\{T_2\}$. The hope is that minimizing $vol\{T_2\}$ minimizes $vol\{T_1 \cap T_2\}$ and hence minimizes the set $\mcl M_A$ which is the preimage of $T_1 \cap T_2$ under the map $z_d$.

In the next section, we will extend this approach to outer semialgebraic set representations of unions of semialgebriac sets. We modify the optimization problem \eqref{opt: monomials to constrain points} to a problem that is solved by a matrix $A \in S_{++}^N$ that defines a sublevel set $L(z_d(x)^T A z_d(x),1)$ that is constrained to contain a union of some semi-algebraic sets. Moreover the sublevel set $L(z_d(x)^T A z_d(x),1)$ heuristically has minimum volume as a $\log \det$ objective function is included.
\begin{align} \label{opt: monomials to constrain sets}
 \min_{A \in S^N_{++}} \{ & -\log \det A \}\\ \nonumber
& \text{subject to: } z_d(x)^T A z_d(x) \le 1 \text{ for } x \in S \nonumber
\end{align}
Where $S= \cup_{i=1}^{m}S_i$, $S_i=\{x \in \R^n: g_{i,1}(x) \le 0,....,g_{i,l_i}(0) \le 0 \}$ and $N=\dim\{z_d\}$.

\subsection{Tractable SOS Tightening}
The optimization problem \eqref{opt: monomials to constrain sets} is currently not a tractable optimization problem. This is because determining whether a polynomial is globally positive ($f(x)>0$ $\forall x \in \R^n$) is NP-hard \cite{Blum_1998}. However it can be shown testing if a polynomial is Sum-of-Squares (SOS) is equivalent to solving a semidefinite program (SDP). Although not all positive polynomials are SOS, this gap can be made arbitrarily small \cite{Hilbert_1888}.

To avoid cumbersome notation we will not state the SDP resulting from the SOS tightening explicitly. However the constraints we give can be enforced using software such as SOSTOOLS \cite{Prajna_2002} that will reformulate the problem as an SDP. Using efficient primal-dual interior point methods for SDP's we are able to solve such proposed problems \cite{Nesterov_1994}.

We now give necessary and sufficient conditions for testing if a polynomial is positive over a semialgebriac set.

\begin{thm}[\cite{Putinar_1993}] \label{thm: Psatz}
Consider the semialgebriac set $X = \{x \in \R^n: g_i(x) \ge 0 \text{ for } i=1,..k\}$. Further suppose $\{x  \in \R^n : g_i(x) \ge 0 \}$ is compact for some $i \in \{1,..,k\}$. If the polynomial $f: \R^n \to \R$ satisfies $f(x)>0$ for all $x \in X$, then there exists SOS polynomials $\{s_i\}_{i \in \{1,..,m\}} \subset \sum_{SOS}$ such that,
\begin{equation*}
f- \sum_{i=1}^m s_ig_i \in \sum_{SOS}.
\end{equation*}
\end{thm}

We now propose a tightening of \eqref{opt: monomials to constrain sets} to a convex SOS program:
\begin{align} \label{opt: SOS to contain sets}
& \min_{A \in S^N_{++}} \{ -\log \det A \} \quad \text{subject to,} \\ \nonumber
 &(1- z_d^T A z_d) - \sum_{j=1}^{l_1} s_{i,j}g_{i,j} \in \sum_{sos}  \quad \forall i\in \{1,...,m\} \\ \nonumber
& \hspace{1.5cm} s_{i,j} \in \sum_{sos} \quad \forall i,j
\end{align}
where $N=\dim\{z_d\}$.

Using Theorem \ref{thm: Psatz} we see the constraints of this optimization problem ensure $ 1- z_d(x)^T A z_d(x) \ge 0 \text{ for } x \in S= \cup_{i=1}^{m}\{x \in \R^n: g_{i,1}(x) \le 0,....,g_{i,l_i}(0) \le 0 \}$. That is any solution of \eqref{opt: SOS to contain sets} is feasible for \eqref{opt: monomials to constrain sets}.

Moreover by adding the constraint $\nabla^2 (z_d^T A z_d) \in \sum_{SOS}$ to \eqref{opt: SOS to contain sets} we can ensure the function $V(x)=z_d(x)^T A z_d(x)$ is convex and thus its 1-sublevel set is also convex.

\section{Numerical examples: Representing the Union of Semialgebriac Sets as a Single Sublevel Set} \label{sec: Numerical examples of set containment}
Next we will give two numerical examples that show for a given union of semi-algebraic sets, $S$, we can use the SOS program \eqref{opt: SOS to contain sets} to find a function $V: \R^n \to \R$ such that $L(V,1)$ is an outer approximation for $S$. Moreover, as described at the end of the previous section, we are able to constrain $L(V,1)$ to be convex. For these examples \eqref{opt: SOS to contain sets} was solved using SOSTOOLS \cite{Prajna_2002}, to reformulate the problem into an SDP, and SDPT3 \cite{sdpt3_2003}, to solve the resulting SDP.

\begin{figure} 	
	\flushleft
	\includegraphics[scale=0.56]{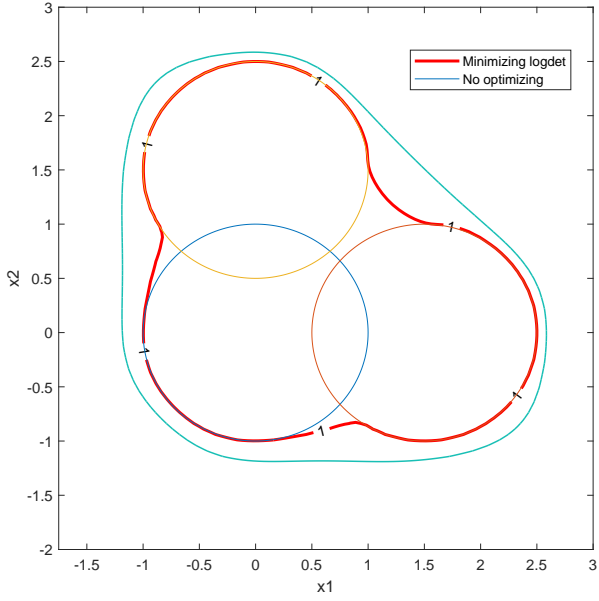}
	\caption{A non-convex outer set approximation of the union of three overlapping semialgebriac sets.} \label{fig:non convex balls}
\end{figure}

	\begin{figure} 	
				\flushleft
		\includegraphics[scale=0.55]{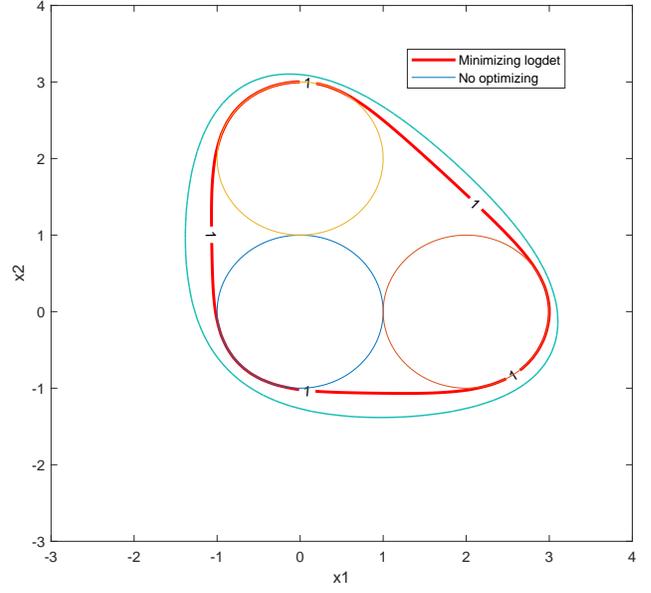}
		\caption{A convex outer set approximation of the union of three semialgebriac sets.} \label{fig:convex balls}
	\end{figure}

Figure \ref{fig:non convex balls} shows the output of the SOS program \eqref{opt: SOS to contain sets} for $d=4$ when we are trying to contain $S= \cup_{i=1}^{3}S_i$ where,
\begin{align*}
S_1 & =\{(x_1,x_2): x_1^2 + x_2^2 \le 1\}\\
S_2 & =  \{(x_1,x_2): (x_1-1.5)^2 + (x_2)^2 \le 1 \} \\
S_3 & =\{ (x_1,x_2): (x_1)^2 + (x_2-1.5)^2 \le 1 \}.
\end{align*}

Figure \ref{fig:convex balls} shows the output of the SOS program \eqref{opt: SOS to contain sets} with an added convexity constraint for $d=3$ when we are trying to contain $S= \cup_{i=1}^{3}S_i$ where
\begin{align*}
S_1 & =\{(x_1,x_2): x_1^2 + x_2^2 \le 1\}\\
S_2 & =  \{(x_1,x_2): (x_1-2)^2 + (x_2)^2 \le 1 \} \\
S_3 & =\{ (x_1,x_2): (x_1)^2 + (x_2-2)^2 \le 1 \}.
\end{align*}

In our experience as $d$ increases we are able to get better outer approximations of $S$. However for large $d$ numerical errors may dominate.

\section{Outer Approximation of Attractors} \label{sec: outeer approx of attractors}
 In this section we would like to find an outer approximation of an attractor of a dynamical system. We do this by considering an optimization problem of the form \eqref{opt: general set countainment}. As in Section \ref{Section: Outer Approximation of sets} we propose a convex optimization problem, similar to \eqref{opt: SOS to contain sets}. Unlike in \eqref{opt: SOS to contain sets} Lyapunov theory is required to ensure the attractor is contained in our sublevel set approximation.

\subsection{Background: Nonlinear Ordinary Differential Equations}
In this paper we are interested in dynamical systems described by ODE's of the form:
\begin{align} \label{eqn: ODE}
\dot{x}(t)=f(x(t))\\ \nonumber
x(0)=x_0
\end{align}
where $f:\R^n \to \R$ and $x_0 \in \R^n$.\\
Throughout this paper we will assume the existence and uniqueness of solutions of ODE's of the form \eqref{eqn: ODE}.

\begin{defn}
	We say $\phi: \R^n \times \R^+ \to \R^n$ is the \textit{solution map} for \eqref{eqn: ODE} if $\frac{\delta \phi(x,t)}{\delta t}= f(\phi(x,t))$ and $\phi(x,0)=x$ for all $x \in R^n$. Moreover for a set $U \subset \R^n$ we denote $\phi(U,t):=\{\phi(x,t): x \in U\} \subset \R^n$.
\end{defn}

\begin{defn}
For an ODE of the form \eqref{eqn: ODE} we say the set $U \subset \R^n$ is an \textit{invariant set} if $x_0 \in U$ implies $\phi(x_0,t) \in U$ for all $t\ge 0$.
\end{defn}

\begin{defn}
	For an ODE of the form \eqref{eqn: ODE} we say $x_0\in \R^n$ is a \textit{periodic initial condition} if $\exists T>0$ such that,
	\begin{equation*}
	\phi(x_0, t+T)=\phi(x_0,t) \quad \forall t \ge 0.
	\end{equation*}
	Moreover we say the set $\mcl L \subset \R^n$ is a  \textit{periodic orbit} if there exists a periodic initial condition $x_0 \in \R^n$ such that,
	\begin{equation*}
	\mcl L = \{ \phi(x_0,t) \in \R^n : t\ge 0\}.
	\end{equation*}
\end{defn}
The next theorem shows that in two dimensional systems invariant sets must contain either stable equilibrium points or periodic orbits. However this is not the case for higher order systems where trajectories can take non-periodic chaotic paths.
\begin{thm}[Poincare-Bendixson Criterion] \label{thm: Poincare-Bendixson Criterion}
	Consider a second order autonomous system represented by an ODE of the form:
	\begin{align} \label{eqn: 2d system}
	\dot{x}_1(t) = f_1(x_1(t),x_2(t))\\ \nonumber
	\dot{x}_2(t) = f_2(x_1(t),x_2(t))
	\end{align}
	where $f_i: \R \times \R \to \R$ for $i =1,2$.\\
	Suppose the set $M \subset \R^n$ is a closed bounded invariant set for the above ODE that contains no stable equilibrium points. Then $M$ contains a periodic orbit.
\end{thm}

	\begin{defn}
	We say $A \subset \R^n$ is an \textit{attractor set} for the ODE \eqref{eqn: ODE} if for any initial condition $x_0 \in \R^n$, there exist a $T>0$ such that $\phi(x_0,t) \in A$ for all $t > T$. Furthermore we say an attractor is a  \textit{minimal attractor} if it has no proper subsets that are also attractors.
\end{defn}

Using Theorem \ref{thm: Poincare-Bendixson Criterion} we can deduce the only attractor sets possible in two dimensional systems are equilibrium points or periodic orbits. However, higher dimensional systems can posses attractors that contain non-periodic (chaotic) trajectories. The methods proposed in this paper can find outer estimates of both attractors that exhibit periodic behavior, such as the limit cycle in the Van der Poll oscillator, as well as chaotic (strange) attractors, such as the Lorenz attractor.

\subsection{How We Compute an Outer Approximate of an Attractor}
In this section we will derive a heuristic algorithm for computing outer approximations of attractor sets. We will formulate this problem as an optimization problem in the form of \eqref{opt: general set countainment} and use the heuristic methods developed in Section \ref{Section: Outer Approximation of sets}. That is, for a bounded minimal attractor, $A$, of some ODE \eqref{eqn: ODE} we would like to solve

\begin{align} \label{opt: general attractor}
\min_{V \in \sum_{SOS}} \{ & D_V(A,L(V,1)) \}\\ \nonumber
& \text{subject to: } A \subseteq L(V,1). \nonumber
\end{align}

Unlike in Section \ref{Section: Outer Approximation of sets} we don't actually know the form of the set we are trying to approximate. However, in Theorem \ref{thm: Lyapunov invariant sets}, for systems with minimal attractors we will give Lyapunov type conditions for $L(V,1)$ to contain the minimal attractor. We first give some preliminary results used in the proof of the theorem. 

\begin{lem} \label{lem: min attractor contained in attractors}
	Suppose an ODE of the form \eqref{eqn: ODE} has a minimal attractor $A \subset \R^n$. Then, if $B \subset \R^n$ is an attractor for \eqref{eqn: ODE} we must have $A \subseteq B$.
\end{lem}

\begin{proof}
	Suppose for contradiction $A \nsubseteq B$. Since $A$ and $B$ are both attractor sets there exists $T_1, T_2>0$ such that $\phi(x_0,t) \in A$ for all $t>T_1$ and $\phi(x_0,t) \in B$ for all $t>T_2$. Thus for $t>\max \{T_1,T_2\}$ we have, $\phi(x_0,t)\in A \cap B$; proving $A \cap B \ne \emptyset$. Furthermore the same argument shows $A \cap B$ is an attractor set. Since, by assumption, $A \nsubseteq B$ and we have also shown $A \cap B \ne \emptyset$ it follows that $A \cap B$ is proper subset of $A$ contradicting $A$ is a minimal attractor. \end{proof}


\begin{cor} \label{cor: unique attractor}
	A system described by an ODE of the form \eqref{eqn: ODE} admits at most one minimal attractor.
\end{cor}
\begin{proof}
	Suppose $A_1$ and $A_2$ are minimal attractor sets for an ODE of the form \eqref{eqn: ODE}. By Lemma \ref{lem: min attractor contained in attractors} we have $A_1 \subseteq A_2$ and $A_2 \subseteq A_1$; therefore proving $A_1=A_2$.
	\end{proof}

\begin{thm} \label{thm: Lyapunov invariant sets}
	Consider some ODE of the Form \eqref{eqn: ODE}. Suppose there exists $V: \R^n \to \R$ such that,
	\begin{align} \label{eqn: hypothesis of theorem 2}
	& V(x)>0 \text{ for all } x \notin D\\ \nonumber
	 & \nabla{V}(x)^T f(x) < 0 \text{ for all } x \notin D.
	\end{align}
	Then if $\gamma>0$ is such that $D \subset L(V, \gamma)$ we have that $L(V, \gamma)$ is an invariant set. Moreover \eqref{eqn: ODE} has a minimal attractor, $A$, and $A \subseteq L(V,\gamma)$.
\end{thm}

\begin{proof}
Consider $\gamma>0$ such that $D \subset L(V,\gamma)$ we first show $L(V, \gamma)$ is invariant. Let $x_0 \in L(V,\gamma)$ and suppose for contradiction $\exists T_{1}>0$ such that $\phi(x_0,T_1) \notin L(V,\gamma)$. By the continuity of the solution map $\exists T_2<T_1$ such that $\phi(x_0,T_2)=\gamma$ and $\frac{d}{dt}V(\phi(x_0,t))|_{t=T_2}>0$. However since $D \cap \{x: V(x)= \gamma\}= \emptyset$ we have $\nabla{V}(x)^Tf(x)<0$ for any $x \in \{x: V(x)= \gamma\}$. Therefore $\frac{d}{dt}V(\phi(x_0,t))|_{t=T_2}<0$ causing a contradiction.

Next we will show that $ L(V,\gamma)$ is an attractor and thus using Lemma \ref{lem: min attractor contained in attractors} $A \subset L(V,1)$. Because $\gamma>0$ is such that $D \subset L(V,\gamma)$ we have $\nabla{V}(x)^T f(x)<0$ for all $x \notin L(V,\gamma)$; i.e. $V$ is strictly decreasing along trajectories with initial conditions outside $L(V,\gamma)$. Thus for all $x_0 \notin L(V,\gamma)$ there exists $ T_{x_0}>0$ such that $\phi(x_0,t) \in L(V,\gamma)$ for all $t>T_{x_0}$; proving $ L(V,\gamma)$ is an attractor. \end{proof}


For given $D=\{x \in \R^n: g(x)\ge 0\}$ and $d \in \N$, we now propose an SOS program that heuristically solves \eqref{opt: general attractor} for an ODE \eqref{eqn: ODE} with a bounded minimal global attractor.
\begin{align} \label{opt: SOS to find attractors}
 \min_{X \in S^N_{++}} \{ & -\log \det X\}\\ \nonumber
& \text{subject to: } V(x)=z_d(x)^T X z_d(x)\\ \nonumber
& s_1, s_2 \in \sum_{SOS} \\ \nonumber
& (1-V) - s_1 g \in \sum_{SOS}  \\ \nonumber
& -\nabla V^T f +s_2 g \in \sum_{SOS} .
\end{align}
Here the constraint $(1-V) - s_1 g \in \sum_{SOS} $ implies $D \subseteq L(V,1)$ by Theorem \ref{thm: Psatz}. The constraint $-\nabla V^T f +s_2 g \in \sum_{SOS} $ implies $\nabla{V}(x)^Tf(x) \le 0$ for all $x \notin D$ by Theorem \ref{thm: Psatz}. Then using Theorem \ref{thm: Lyapunov invariant sets} the constraints therefore imply $L(V,1)$ is an invariant set. Moreover if $\nabla{V}(x)^T f(x) < 0$ for all $x \notin L(V,1)$ then $L(V,1)$ contains the minimal attractor, $A$. In Section \ref{Section: Outer Approximation of sets} it is seen that minimizing the objective function in \eqref{opt: SOS to find attractors} heuristically minimizes $D_V(A,L(V,1))$; reducing the "distance" between the sets $A$ and $L(V,1)$.

\textbf{Note on selection of $g$ in \eqref{opt: SOS to find attractors}:}
	In the optimization problem \eqref{opt: SOS to find attractors} if we choose $D \subset A$, where $A$ is the minimal attractor, the problem becomes infeasible. This is because assuming we are able to find a such a feasible $V \in \sum_{SOS}$ then by the continuity of $V$ we would be able to find $\gamma>0$ such that $A \nsubseteq L(V,\gamma)$ and $D \subset L(V,\gamma)$. This contradicts Theorem \ref{thm: Lyapunov invariant sets}; that is if $V$ satisfies \eqref{eqn: hypothesis of theorem 2} and $D \subset L(V,\gamma)$ then $A \subseteq L(V,\gamma)$.
	
	We also don't want to select $D$ such that $A \subset D$. In the optimization problem \eqref{opt: SOS to find attractors} $L(V,1)$ is constrained so that $D \subset L(V,1)$. Therefore $L(V,1)$ will capture the shape of the set $D$ and not $A$.
	
	Ideally $D$ should be such that $A \nsubseteq D$ and $D \nsubseteq A$ to allow \eqref{opt: SOS to find attractors} to be feasible and $L(V,1)$ to show the shape of $A$. In our numerical results the set $D$ is carefully chosen using trajectory simulation results. Alternatively, if the approximate shape of the attractor is unknown, we have found that the best choice is generally $D=B(0,r)$ where $r>0$ is the smallest $r$ such that $B(0,r) \not\subset A$ (which implies infeasibility of~\eqref{opt: SOS to find attractors}). Bisection can be combined with feasibility of~\eqref{opt: SOS to find attractors} to find the smallest such $r>0$ such that \eqref{opt: SOS to find attractors} is feasible. A similar bisection method was used in \cite{Jones_2017}.

\section{Numerical Examples: Representing Attractor Sets as a Single Sublevel Set} \label{sec: Numerical Examples of Attractor Approximation}
In this section we will present the results of solving the optimization problem \eqref{opt: SOS to find attractors} for two dynamical systems, the Van der Poll oscillator and the Lorentz attractor. For these examples \eqref{opt: SOS to find attractors} was solved using SOSTOOLS \cite{Prajna_2002}, to reformulate the problem as an SDP, and SDPT3 \cite{sdpt3_2003}, to solve the resulting SDP.
	\begin{figure} 	
		\flushleft
		\includegraphics[scale=0.56]{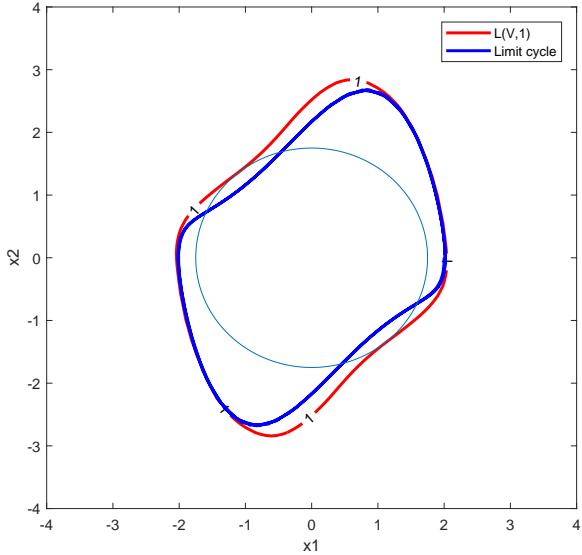}
		\caption{A sublevel set outer approximation of the Van der Pol limit cycle.} \label{fig:Van der pol}
	\end{figure}

\subsection{An Outer Approximation of the Limit Cycle of the Van der Pol Oscillator} \label{sec: Van der pol}
Consider the Van der Pol oscillator defined by the ODE:
\begin{align} \label{eqn: van der pol ode}
\dot{x}_1(t) & = x_2(t)\\ \nonumber
\dot{x}_2(t) & = -x_1(t) + x_2(t)(1- x_1^2(t)).
\end{align}

We applied the proposed method by solving the optimization problem \eqref{opt: SOS to find attractors} for $d=4$ and $D=\{x \in \R^2: x_1^2 +x_2^2 - 1.75^2 \ge 0\}$; where $D$ was hand picked based on simulated trajectory data. The results are displayed in Figure \ref{fig:Van der pol}. The limit cycle of \eqref{eqn: van der pol ode}, represented by the blue curve, was approximately found by forward-time integrating \eqref{eqn: van der pol ode} at an initial starting position close to the limit cycle. The set $D$ was selected such that it only contains part of the limit cycle of \eqref{eqn: van der pol ode}. As expected the boundary of $L(V,1)$ follows tightly across the boundary of the union of $D$ and the limit cycle.

\subsection{An Outer Approximation of the Attractor of the Lorenz Attractor} \label{sec: Lorenz}
We now consider a three dimensional dynamical system that exhibits chaotic characteristics under certain parameters. Consider the Lorenz attractor defined by the ODEs:
\begin{align} \label{eqn: Lorenz attactor}
\dot{x}_1(t) & = \sigma ( x_2(t) - x_1(t) ) \\ \nonumber
\dot{x}_2(t) & = \rho x_1(t) - x_2(t) - x_1(t)x_3(t) \\ \nonumber
\dot{x}_3(t) & = x_1(t)x_2(t) - \beta x_3(t)
\end{align}

\begin{lem} \label{lem: equalibrium points of Lorenz}
	The Lorenz attractor \eqref{eqn: Lorenz attactor} has three equilibrium points at $(0,0,0)^T$ and $( \pm \sqrt{ \beta( \rho -1)}, \pm \sqrt{ \beta( \rho -1)}, \rho -1  )^T$.
\end{lem}
\begin{proof}
	The equilibrium points can be found by writing \eqref{eqn: Lorenz attactor} in the form $\dot{x}(t) = f(x(t))$ and solving the equations $f(x)=(0,0,0)^T$.	
\end{proof}

Throughout this paper we will only consider the case $ \sigma = 10$, $\rho=28$ and $\beta=\frac{8}{3}$ as Lorenz did in 1963. During our numerical results we make a coordinate change so the attractor is located in a unit box by defining
\begin{align} \label{eqn: Lorentz coordinate change}
\bar{x}_1 & :=  50 x_1  \\ \nonumber
\bar{x}_2  & := 50 x_2 \\ \nonumber
\bar{x}_3  & := 50 x_3 + 25  .
\end{align}

Figure \ref{fig:Lorentz_trajectory} shows the approximate shape of the Lorenz attractor under the change of coordinates \eqref{eqn: Lorentz coordinate change} found by forward-time integrating \eqref{eqn: Lorenz attactor} at an initial starting position of $(0,1,1.05)^T$.

	\begin{figure} 	
		\flushleft
		\includegraphics[scale=0.6]{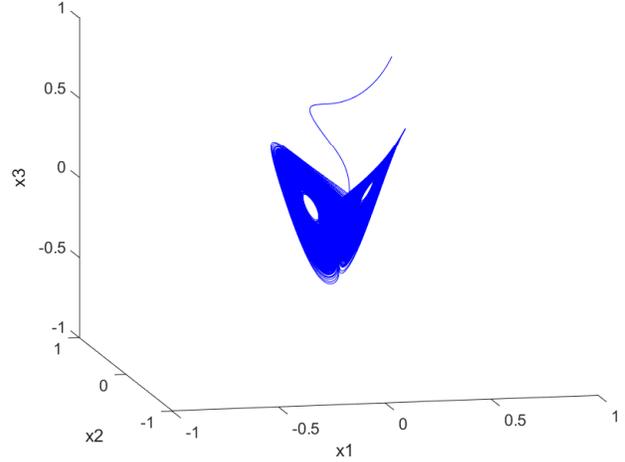}
		\caption{The shape of a trajectory from the Lorenz system found approximately by forward time integration.} \label{fig:Lorentz_trajectory}
	\end{figure}

Figure \ref{fig:lorenz_d_4_elipse_z_c_minus_0_3_x_r_0.6_y_r_0.01_z_r_005_60_degree} shows the boundary of the set $L(V,1)$, represented as the red shell, where $V: \R^3 \to \R$ solves the optimization problem \eqref{opt: SOS to find attractors} for the ODE \eqref{eqn: Lorenz attactor} under the change of coordinates \eqref{eqn: Lorentz coordinate change}. Here $d=4$ and $D$ is an elipsoid centered at $(0,0,-0.3)$ rotated by $60^{\circ}$ given by,
\begin{align} \label{eqn: D Lorentz}
D= \bigg\{ x & \in \R^3:  \frac{(\cos(\theta)x_1 - \sin(\theta) x_2 - c_1)^2}{r_1^2}  \\ \nonumber
&   + \frac{(\sin(\theta)x_1 + \cos(\theta) x_2 -c_2)^2}{r_2^2} + \frac{(x_3 - c_3)^2}{r_3^2} \le 1  \bigg\},
\end{align} where $r_1 = 0.6$, $r_2=0.01$, $r_3=0.05$, $c_1=0$, $c_2=0$, $c_3=-0.3$ and $\theta = 60^{\circ}$.

A sample trajectory of \eqref{eqn: Lorenz attactor} using the coordinates \eqref{eqn: Lorentz coordinate change}, represented by the blue curve in Figure \ref{fig:lorenz_d_4_elipse_z_c_minus_0_3_x_r_0.6_y_r_0.01_z_r_005_60_degree}, was approximately found by forward-time integrating \eqref{eqn: Lorenz attactor} at an initial starting position $(0,1,1.05)^T$. As expected the trajectory is attracted and travels inside the set $L(V,1)$ providing numerical evidence $L(V,1)$ contains the attractor.



		\begin{figure} 	
			\flushleft
			\includegraphics[scale=0.6]{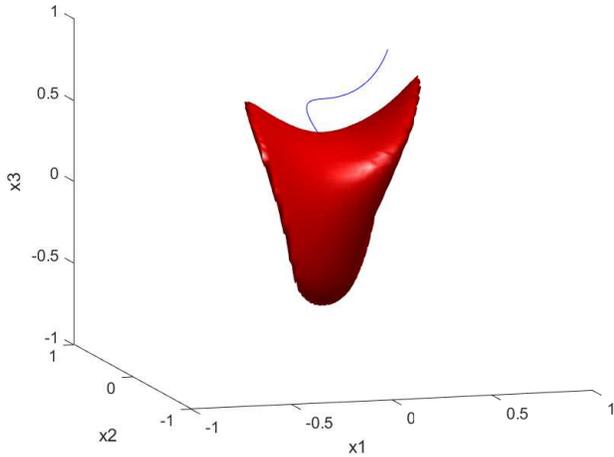}
			\caption{A sublevel set outer approximation of the Lorenz attractor.} \label{fig:lorenz_d_4_elipse_z_c_minus_0_3_x_r_0.6_y_r_0.01_z_r_005_60_degree}
		\end{figure}

\section{Conclusion} \label{sec:conclusion}
We have illustrated a method for finding optimal semialgebraic representations of unions and intersections of semi-algebraic sets with a single sublevel set of an SOS polynomial. We have shown how an objective function based on the determinant heuristically minimizes the volume of sublevel sets and can improve these outer approximations. Furthermore we have applied our methods to finding attractors of nonlinear systems using Lyapunov theory. Outer approximations for the attractors for the Van der Pol and Lorenz system were given. Our numerical examples demonstrate how our method can reveal the shape and properties of attractor sets associated with nonlinear differential equations.

\section{Future work} \label{sec:future work}
We will consider the generalization of the determinant to non-linear algebra. Namely hyper-determinants, discriminants and resultants \cite{Gelfand_1994} \cite{Dolotin_2008} \cite{Morozov_2009}. As in the linear case \eqref{eqn: det is ratio of volumes} we seek to derive a relationship between the volume of sets of the form $L(z_d(x)^T A z_d(x),1)$ and a convex function based on the generalization of the determinant of the SOS polynomial $z_d(x)^T A z_d(x)$.

\bibliographystyle{unsrt}
\bibliography{bib}
\end{document}